\documentclass[12pt]{article}
\usepackage{verbatim}
\usepackage{latexsym}
\usepackage{amsfonts}
\usepackage{amsmath}
\usepackage{amsthm}

\newtheorem{Thm}{Theorem}[section]
\newtheorem{Lem}[Thm]{Lemma}
\newtheorem{Cor}[Thm]{Corollary}

\newcommand{\Oh}{\mathnormal{O}}
\newcommand{\fr}{f_\mathrm{rings}}

\begin{document}

\title{The Enumeration of Finite Rings}
\author{Simon R. Blackburn\\
Department of Mathematics\\
Royal Holloway University of London\\
Egham, Surrey TW20 0EZ, United Kingdom\\
\texttt{s.blackburn@rhul.ac.uk}\\
\and
K.~Robin McLean \\ 
Department of Mathematical Sciences\\
University of Liverpool\\
Liverpool L69 7ZL, United Kingdom
\\
\texttt{krmclean@liv.ac.uk}}
\maketitle

\begin{abstract}
Let $p$ be a fixed prime. We show that the number of isomorphism classes of finite rings of order $p^n$ is $p^\alpha$, where $\alpha=\frac{4}{27}n^3+\Oh(n^{5/2})$. This result was stated (with a weaker error term) by Kruse and Price in 1969; a problem with their proof was pointed out by Knopfmacher in 1973. We also show that the number of isomorphism classes of finite commutative rings of order $p^n$ is $p^\beta$, where $\beta=\frac{2}{27}n^3+\Oh(n^{5/2})$. This result was stated (again with a weaker error term) by Poonen in 2008, with a proof that relies on the problematic step in Kruse and Price's argument. 
\end{abstract}


\section{Introduction}

For a positive integer $N$, let $\fr(N)$ be the number of (isomorphism classes of) finite rings of cardinality $N$. (We do not assume that our rings have a multiplicative identity.) What can be said about this function? 

Write $N=\prod_{i=1}^t p_i^{n_i}$, where the integers $p_i$ are distinct primes. It is not hard to see that a ring $R$ of cardinality $N$ may be uniquely written as a direct sum
$R=\bigoplus_{i=1}^t R_i$, where $R_i$  is a ring of order $p_i^{n_i}$,
and hence
$\fr(N)=\prod_{i=1}^t\fr(p_i^{n_i})$.
So we may specialise to the case when $N$ is a prime power. For the rest of this paper, we assume that $N=p^n$, where $p$ is prime. 

It seems very hard to provide exact values for $\fr(p^{n})$, so it is natural to ask about asymptotic enumeration. We think of $p$ as being fixed, and~$n$ increasing: How fast does this function grow? About fifty years ago, an interesting paper by Kruse and Price \cite{K & P} addressed this question, taking inspiration from the then-recent enumeration of finite $p$-groups by Graham Higman~\cite{Higman} and C.C.~Sims~\cite{Sims}. Kruse and Price~\cite[Corollary~5.9]{K & P} (see also \cite[Chapter~V]{K & P book}) state that the number of (isomorphism classes of) finite rings of cardinality $p^{n}$ is $p^{\alpha}$, where $\alpha = \frac{4}{27}n^{3} + \Oh(n^{8/3})$. The structure of their argument can be summarised as follows. For the lower bound, they construct $p^{\frac{4}{27}n^{3} -\Oh(n^2)}$ rings of cardinality $N$ by taking quotients of a certain $\mathbb{F}_p$-algebra $A$ of nilpotency class $2$ on $r$ generators, where $r\approx 2n/3$. (Here $\mathbb{F}_p$ is the finite field of order~$p$.) It is possible to show that $A$ has about $p^{\frac{4}{27}n^3}$ ideals of index~$p^n$, and it can be shown that each isomorphism class of rings appears at most~$p^{\Oh(n^2)}$ times as a quotient by an ideal of this form. This approach is inspired by Higman's lower bound for the number of isomorphism classes of groups of order $p^n$. To provide a corresponding upper bound, Kruse and Price divide the problem into three steps:
\begin{enumerate}
\item Reduce to the case of $\mathbb{F}_p$-algebras.
\item Provide an upper bound on the number of nilpotent $\mathbb{F}_p$-algebras, using techniques inspired by Sims' enumeration of (nilpotent) groups of order~$p^n$.
\item Show that once the isomorphism class of the (necessarily nilpotent) Jacobson radical $J(R)$ is fixed, there are few possibilities for the algebra $R$ itself. There is an interesting parallel with Pyber's (much later) results on group enumeration~\cite{Pyber} here: Pyber shows that the number of isomorphism classes of groups whose Sylow subgroups have been chosen (up to isomorphism) is small. 
\end{enumerate}

The first step of this approach is ingenious but, sadly, flawed: In 1973, Knopfmacher~\cite[Page~169]{Knopfmacher} already points out that the reduction does not work. We provide counterexamples to this reduction in the appendix below.

The main aim of our paper is to show how the upper bound in~\cite{K & P} may nevertheless be established, by providing analogues of the last two steps above that work without the initial reduction to algebras. First, in Section~\ref{Nilpotent case}, we provide an upper bound on the number of nilpotent rings, rather than just nilpotent $\mathbb{F}_p$-algebras, of cardinality $p^n$. This uses techniques similar to those in Sims~\cite{Sims} and in Kruse and Price~\cite{K & P} (but see below). Secondly, in Section~\ref{General case}, we show that once the isomorphism class of the (nilpotent) Jacobson radical $J(R)$ is fixed, there are few possibilities for the ring $R$. Our argument is (and needs to be) rather different from the argument in~\cite{K & P}.

In fact, our arguments in Section~\ref{Nilpotent case} make use of a trick due to Craig Seeley and M.F.~Newman~\cite{SeeleyNewman} (see Blackburn, Neumann and Venkataraman~\cite[Chapter~5]{BlackburnNeumann} for details) which was originally applied to improve error terms in group enumeration. Adapting this trick to the situation of rings, we are able to show (Theorem~\ref{general bound}) that the number of rings of cardinality $p^{n}$ is $p^{\alpha}$, where
\[
\alpha = \tfrac{4}{27}n^{3} + \Oh(n^{5/2}).
\]
We note (see Theorem~\ref{ring with one bound}) that the same statement holds when enumerating rings with identity.

More recently, Poonen~\cite[Section~11]{Poonen} has used the same problematic initial reduction to algebras~\cite[Lemma~11.1]{Poonen} in order to enumerate commutative rings. In Section~\ref{sec:commutative}, we provide a proof of this enumeration also, with an improved error term. We show (Theorem~\ref{commutative main}) that the number of commutative rings of cardinality $p^n$ is $p^\beta$, where
\[
\beta= \tfrac{2}{27}n^{3} + \Oh(n^{5/2}).
\]
The reduction to nilpotent rings in Section~\ref{General case} can be used unchanged in the commutative situation. However, the enumeration of nilpotent commutative rings requires extra work. As before, our results allow us to enumerate commutative rings with identity; see Theorem~\ref{commutative ring with one bound}. 

Finally, in the appendix, we exhibit counterexamples to the proof of Kruse and Price.

\paragraph{Dedication} We dedicate this paper to the memory of Peter M. Neumann. Peter was responsible for initiating the authors' collaboration on this paper, bringing us together after Robin wrote to him with counterexamples to Kruse and Price's argument and early ideas for avoiding the problematic reduction to $\mathbb{F}_p$-algebras. Simon was one of Peter's many D.Phil.\  students, and he would like to acknowledge the lasting influence of Peter's knowledge, advice and encouragement throughout his career. He will be greatly missed.

\section{Finite nilpotent rings}
\label{Nilpotent case}

Let $R$ be a nilpotent ring of order $p^n$, and define $\overline{R}=R/{pR}$. Roughly speaking (using the insight of Sims) either $\overline{R}/\overline{R}^3$ has very restricted structure (so there are few possibilities for it), or there exists a small set of elements of $R$ whose images in $\overline{R}$ generate a sub-algebra containing $\overline{R}^2$. The multiplicative structure of this small set determines most of the structure of $R$, allowing us to provide a tight enumeration of nilpotent rings. The arguments in this section are similar to~\cite{K & P}, although we use an idea of Seeley and Newman to improve the error term in our enumeration. 

We begin with two structural results (Lemmas~\ref{cube zero} and~\ref{alpha}) concerning $\mathbb{F}_p$-algebras $R$, before proving the main theorem of this section (Theorem~\ref{main nilpotent}). Define the \textit{Sims dimension} of the $\mathbb{F}_p$-algebra $R$ to be the least dimension of $S/R^{2}$ as $S$ ranges over the subalgebras $S$ of $R$ such that $S^{2} = R^{2}$. The following result is Lemma 5.5 of \cite {K & P}.

\begin{Lem} \label{cube zero}
Let $R$ be a nilpotent $\mathbb{F}_{p}$-algebra such that $r = \dim(R/R^{2})$, $s$ is the Sims dimension of $R$, $t = \dim(R^{2})$ and $R^{3} = 0$. If $S$ is a subalgebra of $R$ that contains $R^{2}$, then
\begin{displaymath}
\dim(S^{2}) - \dim(S/R^{2}) \leq t - s + 1.
\end{displaymath}
\end{Lem}

\begin{Lem} \label{alpha}
Let $r, s, t$ be positive integers and let $\alpha(r, s, t)$ be a real number. Suppose that $p^{\alpha(r, s, t)}$ is the number of (isomorphism classes of) $\mathbb{F}_{p}$-algebras $R$ such that 
$r = \dim(R/R^{2})$, $s$ is the Sims dimension of $R$, $t = \dim(R^{2})$ and $R^{3} = 0$, then 
\begin{eqnarray}
\alpha(r, s, t) & \leq & r^{2}(t - s) + \mathnormal{O}((r + t)^{8/3})  \label{old_inequal} \\
 \textit{and} \hspace{5 mm} \alpha(r, s, t) & \leq & r^{2}(t - s) + \tfrac{1}{2}rst + \mathnormal{O}((r + t)^{5/2}). \label{new_inequal}
\end{eqnarray}
\end{Lem}
\begin{proof} The inequality (\ref{old_inequal}) is \cite[Theorem~5.4]{K & P}.
\par To prove (\ref{new_inequal}), let $x_{1}, x_{2}, \ldots, x_{r}$ be elements of $R$ that map onto a basis for $R/R^{2}$. The ring $R$ is determined up to isomorphism by the $r^{2}$ products $x_{i}x_{j}$ given by $1 \leq i, j \leq r$. Now $R^{2}$ contains $p^{t}$ elements, so there are at most $p^{r^{2}t}$ ways of choosing all our $x_{i}x_{j}$. Hence $\alpha(r, s, t) \leq r^{2}t$. If $s \leq 2r^{\frac{1}{2}} + 1$, then $-r^2s+\tfrac{1}{2}rst=\mathnormal{O}((r + t)^{5/2})$, and so this inequality for $\alpha(r, s, t)$ gives (\ref{new_inequal}).
\par We may now suppose that $s > 2r^{\frac{1}{2}} + 1$. Let $f = \lfloor r^{\frac{1}{2}}\rfloor$ be the greatest integer not exceeding $r^{\frac{1}{2}}$ and let $g = \lceil r/f\rceil$ be the least integer not less than $r/f$. Writing $\langle v_1,v_2,\ldots ,v_k\rangle$ for the $\mathbb{F}_{p}$-subspace spanned by elements $v_1,v_2,\ldots ,v_k\in R$, we may define $\mathbb{F}_{p}$-subspaces $V_{1}, V_{2}, \ldots, V_{g}$ of $R$ by
\begin{align*}
V_{i} & = \langle x_{(i-1)f+1}, x_{(i-1)f+2}, \ldots, x_{(i-1)f+f}\rangle  \hspace{2mm} \textrm{for} \hspace{2mm}1 \leq i < g \\
\textrm{and} \hspace{2mm} V_{g} & = \langle x_{(g-1)f+1}, x_{(g-1)f+2}, \ldots, x_{r}\rangle.
\end{align*}
Taking $S = V_{i} + V_{j} + R^{2}$ in Lemma \ref{cube zero} and noting that $R^{3} = 0$, we get
\begin{align}
\dim(V_{i} + V_{j})^{2} & \leq \dim(V_{i} + V_{j}) + t - s + 1 \nonumber \\
{} & = \dim V_{i} + \dim V_{j} + t - s + 1  \nonumber \\
{} & \leq 2f + t - s + 1  \label{third}  \\
{} & \leq 2r^{\frac{1}{2}} + t - s + 1 < t. \label{fourth}
\end{align}
Let $d = 2f + t - s + 1$. From (\ref{third}) and (\ref{fourth}), there is a $d$-dimensional subspace $W_{ij}$ of $R^{2}$ that contains $(V_{i} + V_{j})^{2}$. Now the number of such subspaces is
\begin{equation} \label{num of subspaces}
\frac{(p^{t} - 1)(p^{t} - p) \ldots (p^{t} - p^{d-1})}{(p^{d} - 1)(p^{d} - p) \ldots (p^{d} - p^{d-1})}.
\end{equation}
When $0 \leq i < d$, we have
\begin{displaymath}
p^{t} - p^{i} < p^{t} \leq p^{t+1} - p^{t} \leq p^{t+1} - p^{t-d+i+1} = p^{t-d+1}(p^{d}- p^{i}),
\end{displaymath}
so
\begin{equation} \label{fraction}
\frac{p^{t} - p^{i}}{p^{d} - p^{i}} \leq p^{t-d+1}.
\end{equation}
From (\ref{num of subspaces}) and (\ref{fraction}), the number of ways of choosing each subspace $W_{ij}$ is at most $p^{d(t-d+1)}$, so the number of ways of choosing all $\binom{g}{2}$ subspaces $W_{ij}$ is at most $p$ to the power $\binom{g}{2}d(t - d + 1)$. Once all the $W_{ij}$ have been chosen, there are at most $p^{d}$ choices for each product $x_{i}x_{j}$ with $1 \leq i, j \leq r$. So 
\begin{align*}
\alpha(r, s, t) & \leq \binom{g}{2}d(t - d + 1) + r^{2}d \\
{} & \leq \binom{g}{2}(2f + t - s + 1)(s - 2f) + r^{2}(2f + t - s + 1).
\end{align*} 
Now $f = r^{\frac{1}{2}} + \mathnormal{O}(1)$ and $g = r^{\frac{1}{2}} + \mathnormal{O}(1)$, so
\begin{align*}
\alpha(r, s, t) & \leq \tfrac{1}{2}r(t - s)s + r^{2}(t - s) + \mathnormal{O}((r + t)^{5/2}) \\
{} & \leq r^{2}(t - s) + \tfrac{1}{2}rst - \tfrac{1}{2}rs^{2} + \mathnormal{O}((r + t)^{5/2}) \\
{} & \leq r^{2}(t - s) + \tfrac{1}{2}rst  + \mathnormal{O}((r + t)^{5/2}).\qedhere
\end{align*}
\end{proof}

For a representative $R$ of each of the $p^{\alpha(r,s,t)}$ isomorphism classes of $\mathbb{F}_p$-algebras of the form above, choose a subalgebra $S$ such that $S^2=R^2$ and $\dim S/R^2 = s$. We choose a basis $x_1,x_2,\ldots,x_r,y_1,y_2,\ldots y_t$ of $R$, which we call the \emph{standard basis for $R$}, with the following properties: the elements $x_1+R^2,x_2+R^2,\ldots,x_s+R^2$ form a basis for $S/R^2$; the elements $x_1+R^2,x_2+R^2,\ldots,x_r+R^2$ form a basis for $R/R^2$; the elements $y_1,y_2,\ldots,y_t$ form a basis for $R^2$; each element $y_i$ may be written in the form $y_i=x_kx_\ell$ for some $k,\ell\in\{1,2,\ldots,s\}$. A standard basis exists, since $R$ has cube zero and Sims dimension~$s$. For each element $y_i$, we choose an equality of the form $y_i=x_kx_\ell$ for some $k,\ell\in\{1,2,\ldots,s\}$, and call this the \emph{standard monomial representation} of $y_i$.

\begin{Thm} \label{main nilpotent}
The number of (isomorphism classes of) nilpotent rings of order $p^{n}$ is $p^{\alpha}$, where
\begin{displaymath}
\alpha = \frac{4}{27}n^{3} + \mathnormal{O}(n^{5/2}).
\end{displaymath}
\end{Thm}
\begin{proof} From \cite[Theorem 2.2]{K & P}, the number of $\mathbb{F}_{p}$-algebras of cube zero and order $p^{n}$ is $p^{\alpha'}$, where $\alpha' = 4n^{3}/27 + \mathnormal{O}(n^{2})$. Hence it is sufficient to prove that $p^{\alpha}$ is an upper bound for the number of nilpotent rings of order $p^{n}$.

Let $R$ be a nilpotent ring of order $p^n$, and let $\overline{R} = R/pR$, which is an $\mathbb{F}_{p}$-algebra of order $p^w$ where $1\leq w\leq n$. The additive structure of $R$ is isomorphic to an abelian group of order $p^{n}$ and rank $w$. We fix some standard representative $G$ for this isomorphism class of abelian groups. Let $r = \dim(\overline{R}/\overline{R}^{2})$, let $s$ be the Sims dimension of $\overline{R}$, let $t = \dim(\overline{R}^{2}/\overline{R}^{3})$ and $u = \dim(\overline{R}^{3})$. Let $m$ be the least integer such that $\overline{R}^{m} = 0$. For $h\geq 2$, define $u_h=\dim \overline{R}^h/\overline{R}^{h+1}$.

The number of choices for the isomorphism class of an abelian group $G$ of order $p^n$ is the number of partitions of $n$, which is at most $2^{n-1}<p^n$. There are at most $n+1$ choices for each of $w$, $r$, $s$, $t$ and $m$. The integer $u$ is then determined by $u=w-(r+t)$. We see that $u_h=0$ for $h\geq m$. Now, $(u_2,u_3,\ldots,u_{m-1})$ is a sequence of $m-2$ positive integers that sum to $u+t$, and so there are at most $2^{u+t-1}<p^{n}$ choices for the integers~$u_h$.
  So we have made $(n+1)^5 p^{2n}=p^{\mathnormal{O}(n)}$ choices in all. From now on, we assume that $w$, $r$, $s$, $t$, $m$ and the integers $u_h$ are fixed.

The quotient $\overline{R}/\overline{R}^3$ is an $\mathbb{F}_p$-algebra of order $p^{r+t}$ and cube zero, whose square has order $p^t$ and whose Sims dimension is $s$. We choose one of the $p^{\alpha(r,s,t)}$ isomorphism classes of $\overline{R}/\overline{R}^3$, and from now on we assume this choice is also fixed.

Let $\overline{x}_{1}, \ldots, \overline{x}_{r}\in \overline{R}$ map onto the first $r$ elements of the standard basis for $\overline{R}/\overline{R}^3$. Let $\overline{S}$ be the subring of $\overline{R}$ generated by the first $s$ elements $\overline{x}_{1}, \ldots, \overline{x}_{s}$. Because our basis of $\overline{R}/\overline{R}^3$ is standard, $\overline{S}^{2}+\overline{R}^3 = \overline{R}^{2}$. By~\cite[Theorem~5.2]{K & P}, using the fact that $\overline{R}^m=0$, we see that  
\begin{equation} \label{powers of S}
\overline{S}^{i} = \overline{R}^{i} \quad \textrm{ for all integers} \ i \geq 2.
\end{equation}

For $i\geq 2$, let $d(i)=\dim(\overline{R}^{i})$. So  $d(2)=u+t$, and $d(i) = \sum_{h=i}^{m-1}u_h$ when $i\geq 3$. We will choose a basis $\overline{e}_{1}, \ldots, \overline{e}_{u+t}$ of $\overline{R}^{2}$ in such a way that $\overline{e}_{1}, \ldots , \overline{e}_{d(i)}$ is a basis of $\overline{R}^{i}$ for all $i\geq 2$, and each element of the basis is a monomial in $\overline{x}_{1}, \ldots, \overline{x}_{s}$. We do this as follows. We first choose monomials $\overline{e}_{u+1}, \ldots, \overline{e}_{u+t}$ in $\overline{x}_{1}, \ldots, \overline{x}_{s}$ whose images in $\overline{R}/\overline{R}^3$ are $y_1,y_2,\ldots,y_t$ from the standard basis for $\overline{R}/\overline{R}^3$; we use the standard monomial representations of the standard basis to do this. Since~\eqref{powers of S} holds, the set of elements $\{\overline{x}_i\overline{e}_j:1\leq i\leq s\text{ and }u+1\leq j\leq u+t\}$ span $\overline{R}^3$ modulo $\overline{R}^4$. We choose a subset $\overline{e}_{d(4)+1}, \ldots , \overline{e}_{d(3)}$ of this set which is a basis of $\overline{R}^3$ modulo $\overline{R}^4$. We continue in this way, choosing $\overline{e}_{d(i+1)+1}, \ldots , \overline{e}_{d(i)}$ to be a basis of $\overline{R}^i$ modulo $\overline{R}^{i+1}$ contained in the set $\{\overline{x}_i\overline{e}_j:1\leq i\leq s\text{ and }d(i)+1\leq j\leq d(i-1)\}$, to produce the basis of the form we require.

We have chosen the integers $w$, $r$, $s$, $t$, $m$, the integers $u_h$, and the isomorphism class of $\overline{R}/\overline{R}^3$. 
Though we do not need this, we remark that $\overline{R}$ is determined up to isomorphism by these choices, together with a knowledge of each of the $r^{2} + s(w - r)$ products 
\begin{eqnarray*}
{} & \overline{x}_{i}\overline{x}_{j} & \textrm{for all $i,j$ such that $1 \leq i, j \leq r$} \\
\textrm{and} & \overline{x}_{i}\overline{e}_{j} & \textrm{for all $i, j$ such that $1 \leq i \leq s$ and $1 \leq j \leq u + t = w - r$}
\end{eqnarray*}
as a linear combination of the $\overline{e}_{k}$. We can see this as follows. First note that the way we have chosen our basis allows us to deduce a representation of each $\overline{e}_j$ as a monomial in $\overline{x}_1,\ldots, \overline{x}_s$ from these $r^{2} + s(w - r)$ products. Then note that a product of two elements of the form $\overline{x}_i$ or $\overline{e}_j$ can be computed from associativity together with the fact that each $\overline{e}_{j}$ is a monomial of at least second degree in $\overline{x}_{1}, \ldots, \overline{x}_{s}$. (We give a similar argument for $R$, in a little more detail, below.)

There are coefficients $\lambda_{i,j,k}; \mu_{i,j,k}; \nu_{i,j,k} \in \mathbb{F}_{p}$ such that
\begin{align*}
\overline{x}_{i}\overline{x}_{j} & = \sum_{k=1}^{u}\lambda_{i,j,k}\overline{e}_{k} + \sum_{k=u+1}^{u+t}\mu_{i,j,k}\overline{e}_{k} \text{ where }1 \leq i, j \leq r,\\
\overline{x}_{i}\overline{e}_{j} & = \sum_{k=1}^{d(k+1)}\nu_{i,j,k}\overline{e}_{k} \text{ where } 1 \leq i \leq s,\,\,  \overline{e}_{j} \in \overline{R}^{k}\setminus\overline{R}^{k+1} \text{ and }2\leq h\leq m-2,\\
\overline{x}_{i}\overline{e}_{j} & = 0 \text{ where } 1 \leq i \leq s \text{ and } \overline{e}_{j} \in \overline{R}^{m-1}.
\end{align*}
The values of coefficients $\mu_{i,j,k}$ are determined by the isomorphism class of $\overline{R}/\overline{R}^3$, since the images of $\overline{x}_{1}, \ldots,  \overline{x}_{r},\overline{e}_{u+1}, \ldots, \overline{e}_{u+t}$ in $\overline{R}/\overline{R}^3$ form a standard basis. There are $r^{2}u$ coefficients $\lambda_{i,j,k}$, each of which is an element of $\mathbb{F}_{p}$, so the number of possible values of all the $\lambda_{i,j,k}$ is $p^{r^{2}u}$. For each $h$ in the range $2 \leq h \leq m - 1$, there are $u_{h}$ values of $j$ for which $\overline{e}_{j} \in \overline {R}^{h}\setminus\overline{R}^{h+1}$, so the total number of coefficients $\nu_{i,j,k}$ is
\begin{align*}
s\sum_{h=2}^{m-2}u_{h}d(h + 1) & =  s\sum_{h=2}^{m-2}u_{h}(u_{h+1} + u_{h+2} + \ldots + u_{m-1}) \\
{} & = \tfrac{1}{2}s\big((u_{2} + u_{3} + \cdots + u_{m-1})^{2} - u_{2}^{2} - u_{3}^{2} - \cdots - u_{m-1}^{2}\big)\\
&\leq \tfrac{1}{2}s\left((u_2+u_3+\cdots+u_{m-1})^2-u_2^2\right)\\
{} &= \tfrac{1}{2}s\left((u + t)^{2} - t^{2}\right).
\end{align*}
Hence the number of choices for the $r^{2}$ products $\overline{x}_{i}\overline{x}_{j}$ and the $s(w - r)$ products $\overline{x}_{i}\overline{e}_{j}$ is at most $p^{\beta}$, where 
\begin{align*}
\beta &= r^{2}u + \tfrac{1}{2}s\left((u + t)^{2} - t^{2}\right)\\
&=r^{2}(w - r - t) + \tfrac{1}{2}s\left((w - r)^{2} - t^{2}\right),
\end{align*}
since $u = w - r - t$. By \cite[Corollary 5.3]{K & P}, $s \leq t + 1$, so
\begin{equation} \label{beta}
\beta \leq  r^{2}(w - r - t) + \tfrac{1}{2}s\left((w - r)^{2} - (s - 1)^{2}\right).
\end{equation}
We now fix one of the $p^\beta$ choices for the coefficients $\lambda_{i,j,k}$ and $\nu_{i,j,k}$, thus determining the isomorphism class of $\overline{R}$.

Lift the basis elements $\overline{x}_{1}, \ldots \overline{x}_{r}$ arbitarily to elements $x_{1}, \ldots x_{r}$ of $R$. Each element $\overline{e}_i$ is a monomial in the elements $\overline{x}_{1},\overline{x}_{2},\ldots,\overline{x}_{s}$, as a consequence of the products above. We lift $\overline{e}_i$ to the corresponding monomial $e_i\in R$ in the elements $x_1,x_2,\ldots,x_s$. The elements $e_{1}, \ldots, e_{u+t}, x_{1}, \ldots, x_{r}$ are additive generators of $R$, since $R$ is a group of prime power order under addition and the elements additively generate $R$ modulo $pR$.

We have already chosen an additive group $G$ for $R$. There are at most $p^n$ ways of representing an element $x_i$ or $e_i$ as an element of $G$. There are at most $n$ elements to represent, so there are at most $p^{n^2}$ choices for these representations. After making these choices, addition in the ring (with elements represented as sums of the elements $x_i$ and $e_i$ in $G$) is determined. We choose the values of the $r^{2}$ products $x_{i}x_{j}$ where $1\leq i,j\leq r$ and the $s(w - r)$ products $x_{i}e_{j}$ where $1\leq i\leq s$ and $1\leq j\leq u+t=w-r$. Each of these $r^{2} + s(w - r)$ products is known modulo $pR$ and the number of elements in $pR$ is $p^{n - w}$ so the number of possibilities for these products, after the choices we have fixed above, is at most $p^{\gamma}$, where  
%\begin{equation} \label{gamma}
\[
\gamma = \{r^{2} + s(w - r)\}(n - w).
\]
%\end{equation}

Once we have made these choices, we claim that multiplication in $R$, and hence the isomorphism class of $R$, is determined. To see this, note that left multiplication by $x_i$ is determined by the products above for any $1\leq i\leq s$, since $G$ is generated by the elements $x_j$ and $e_j$. So left multiplication by any monomial in $x_1,x_2,\ldots,x_s$ is determined, by induction on the degree of the monomial. Each element $e_j$ is a monomial in $x_1,x_2,\ldots,x_s$, and (by our choice of elements $e_j$) a suitable monomial may be deduced from the products above. So left multiplication by each element $e_j$ is also determined.  We now show that left multiplication by $x_a$, for $s+1\leq a\leq  r$, is a consequence of the products above. To show this, it is sufficient to consider each product $x_a e_b$ for $1\leq b\leq w-r$. We know that $e_b$ is equal to some monomial $x_{i_1}x_{i_2}\cdots x_{i_k}$ where $k\geq 2$ and $1\leq i_1,i_2,\ldots, i_k\leq s$. So, by associativity, $x_a e_b=(x_a x_{i_1})\cdot (x_{i_2}\cdots x_{i_k})$. The second factor is a monomial in $x_1,x_2,\ldots,x_s$, and so is determined. The first factor is a sum of elements $e_j$ since it lies in $R^2$, and its value is determined by one of the products we have fixed. Since left multiplication by each element $e_j$ is determined, the product $x_a e_b$ is determined. Since addition in the ring is fixed, we see that left multiplication by any sum of the elements $x_i$ and $e_i$, in other words by any element of the ring, is determined using the distributive property of multiplication. So the product of any two elements in our ring is determined, and our claim follows.

To summarise, we have shown that the number of isomorphism classes of rings of order $p^n$ is at most $p^{\delta+\mathnormal{O}(n^2)}$, where $\delta$ is the maximum value of
\begin{multline} \label{delta and alpha}
\alpha(r, s, t) +  \{r^{2} + s(w - r)\}(n - w) + r^{2}(w - r - t)  \\
 + \tfrac{1}{2}s\{(w - r)^{2} - (s - 1)^{2}\}
\end{multline}
over all possible values of $r,s,t$ and $w$. It suffices to show that $\delta\leq \frac{4}{27}n^{3} + \mathnormal{O}(n^{5/2})$. 

If $w=n$ and $u=0$ then $R = \overline{R}$ and $\overline{R}^{3} = 0$, in which case the desired result follows from \cite[Theorem~2.2]{K & P}. So we may assume that $u>0$ or $w<n$. From \cite[Corollary 5.3]{K & P},  $0\leq s\leq t+1$. Hence
\[
r + s \leq r + t + 1 \leq r+t+u+1 = w+1 \leq n+1
\]
where one of the last two inequalities is strict; thus $r + s \leq n$. It is clear that $0 \leq s \leq r$. Now let $x = r/n, y = s/n, z = t/n$ and $v = w/n$. Then 
\begin{equation} \label{triangle}
0 \leq y \leq x \quad \textrm{and} \quad x + y \leq 1.
\end{equation}
Following Seeley and Newman's idea~\cite{SeeleyNewman}, we consider separately the cases $x \leq 3/5$ and $x \geq 3/5$.
\par When $x \leq 3/5$, we substitute the inequality~\eqref{old_inequal} from Lemma \ref{alpha} into~(\ref{delta and alpha}). Allowing the final error term to absorb minor terms, we get
\begin{displaymath}
\delta \leq  \{sw + (r - s)r\}(n - w) + r^{2}(w - r - s) + \tfrac{1}{2}s\{(w - r)^{2} - s^{2}\} + \mathnormal{O}(n^{8/3}) .
\end{displaymath}
So
\begin{align*}
\frac{\delta}{n^{3}} & \leq  (yv + x^{2} - xy)(1 - v) + x^{2}(v - x - y) + \tfrac{1}{2}y((v - x)^{2} - y^{2}) + \mathnormal{O}(n^{-1/3}) \\
 & = x^{2}(1 - x) + \tfrac{1}{2}y\{2 - (x + 1)^{2} - (1 - v)^{2}\} - \tfrac{1}{2}y^{3} + \mathnormal{O}(n^{-1/3}) \\
& \leq   x^{2}(1 - x) + \tfrac{1}{2}y\{2 - (x + 1)^{2}\} - \tfrac{1}{2}y^{3} + \mathnormal{O}(n^{-1/3})\\
&\leq 18/125+\mathnormal{O}(n^{-1/3}),
\end{align*}
where the last inequality follows by applying standard calculus techniques over the region where the constraints $x\leq 3/5$ and~\eqref{triangle} hold. Hence, as $\frac{18}{125} < \frac{4}{27}$, we get $\delta \leq\frac{4}{27}n^{3} + \mathnormal{O}(n^{5/2})$ when $x \leq 3/5$, as desired.

When $x \geq 3/5$ we substitute the inequality~\eqref{new_inequal} from Lemma~\ref{alpha} into~\eqref{delta and alpha}. Allowing the final error term to absorb minor terms, we get
\begin{displaymath}
\delta \leq  \{sw + (r - s)r\}(n - w) + r^{2}(w - r - s) + \tfrac{1}{2}rst + \tfrac{1}{2}s\{(w - r)^{2} - s^{2}\} + \mathnormal{O}(n^{5/2}).
\end{displaymath}
So we find that
\begin{align*}
\frac{\delta}{n^{3}} &\leq (yv + x^{2} - xy)(1 - v) + x^{2}(v - x - y) + \tfrac{1}{2}xyz\\
&\quad + \tfrac{1}{2}y((v - x)^{2} - y^{2}) + \mathnormal{O}(n^{-\frac{1}{2}})\\
&\leq \frac{4}{27}+ \mathnormal{O}(n^{-\frac{1}{2}}),
\end{align*}
with the last inequality following by applying standard calculus techniques over the region where the constraints $x\geq 3/5$, $z\leq 1$, $v\leq 1$ and~\eqref{triangle} hold.

Hence in both cases, whether $x \leq 3/5$ or $x \geq 3/5$, we have 
\begin{displaymath}
\delta \leq \frac{4}{27}n^{3} + \mathnormal{O}(n^{5/2}).\qedhere
\end{displaymath}
\end{proof}

\section{Finite rings in general}
\label{General case}

For a ring $R$, we write $J(R)$ for the Jacobson radical of $R$. The following theorem provides some structure theory for finite rings that we require.
\begin{Thm}
\label{S_exist}
Let $R$ be a finite non-nilpotent ring of $p$-power order, not necessarily with an identity element. There exists a subring $S$ of $R$ such that:
\begin{enumerate}
\item[\textup{(i)}] $S+J(R)=R$,
\item[\textup{(ii)}] $S$ has a multiplicative identity,
\item[\textup{(iii)}] $J(S)=pS$, and
\item[\textup{(iv)}] $S\cap J(R)=pS$.
\end{enumerate}
\end{Thm}

We say that the subring $S$ is a \emph{coefficient ring} of $R$. (This terminology comes from the situation when $R$ has an identity element, as then every element of $R$ can be written as a polynomial in the generators of $J(R)$ with coefficients in $S$. This is not necessarily true in the more general situation.)

\begin{proof}
Let $S$ be a subring of $R$ that is minimal subject to~(i) holding. (Clearly a minimal subring exists, since $R$ is finite and since~(i) holds when $S=R$.) We will show that properties (ii), (iii) and (iv) all hold.

We establish property~(ii) first. The ring $R$ is artinian, as it is finite. Since the Jacobson radical $J(R)$ of $R$ is nilpotent, but $R$ is not nilpotent, $J(R)$ is a proper ideal of $R$. The quotient $R/J(R)$ is semi-simple and non-trivial, and so the Wedderburn--Artin Theorem implies in particular that $R/J(R)$ has a multiplicative identity. By property~(i) this identity may be written in the form $s+J(R)$ for some $s\in S$. Since $J(R)$ is nilpotent, it is a nil ideal. So the idempotent $s+J(R)$ may be lifted to an idempotent $e\in R$. Indeed (see, for example, the proof of~\cite[Proposition~III.8.3]{Jacobson}), $e$ may be taken to be a polynomial in $s$, and so we may assume that $e\in S$. Now $eSe$ is a subring of $S$ with identity $e$. Moreover, since $e+J(R)=s+J(R)$, we see that $e+J(R)$ is the identity in $R/J(R)$ and so $eSe+J(R)=S+J(R)$. By the minimality of $S$, we see that $eSe=S$, and so property~(ii) is established.

We now establish property~(iii), using an approach that is inspired by~\cite[Lemma~1]{Clark}.

Let $R$ have characteristic $p^k$. Then $pR$ is a nil (two-sided) ideal; indeed $pR$ is actually nilpotent, since $(pR)^k=0$. Since $J(R)$ contains all nil ideals~\cite[Theorem I.6.2]{Jacobson}, we see that $pR\subseteq J(R)$. Similarly,
\begin{equation}
\label{pSinJS}
pS\subseteq J(S).
\end{equation} 

Now, $J(R/J(R))$ is trivial~\cite[Theorem~I.2.2]{Jacobson}. Since~(i) holds, the natural map from $S$ to $R/J(R)$ is surjective, and so (see~\cite[Proposition~I.7.1]{Jacobson}) the radical $J(S)$ is mapped into the trivial subring $J(R/J(R))$. Hence
\[
J(S)\subseteq J(R).
\]
In particular, combining with~\eqref{pSinJS}, we find that
\begin{equation}
\label{pSinJR}
pS\subseteq J(R).
\end{equation}

The quotient ring $S/pS$ is an $\mathbb{F}_p$-algebra. By~\eqref{pSinJS}, we find that $J(S/pS)=J(S)/pS$ (see~\cite[Proposition~10.4.3]{Cohn}). Now, the Wedderburn--Malcev Theorem for $\mathbb{F}_p$-algebras~\cite[Theorem~72.19]{CurtisReiner} shows that there exists a subring $T$ of $S$ containing $pS$ with the property that there is an additive decomposition
\begin{equation}
\label{WMdecomp}
S/pS=T/pS\oplus J(S/pS)=T/pS\oplus J(S)/pS.
\end{equation}
We see that
\[
R=S+J(R)=(T+J(S)+pS)+J(R)=T+J(R).
\]
By the minimality of $S$, we find that $T=S$, and so~\eqref{WMdecomp} implies that
\[
J(S)\subseteq pS.
\]
This, together with~\eqref{pSinJS} implies property~(iii).

To establish property~(iv), first note that
$pS\subseteq pR\subseteq J(R)$ and so $pS\subseteq S\cap J(R)$.
Secondly, note that $J(R)$ is a nil ideal of $R$ (as it is nilpotent), so $S\cap J(R)$ is a nil ideal of $S$. Hence
\[
S\cap J(R)\subseteq J(S)=pS,
\]
by~(iii) and because $J(S)$ contains all nil ideals in $S$ (see~\cite[Corollary to Theorem~I.6.2]{Jacobson}). Hence~(iv) follows and the theorem is proved.
\end{proof}

We comment that there is a natural, but less elementary, proof of Theorem~\ref{S_exist}~(iii) using Azumaya's generalized Wedderburn-Malcev theorem (see \cite[Proposition~19]{Azum 72} or \cite[Theorem~33]{Azum 51}). In this approach, we define $Z$ to be the set of integer multiples of $e$ (so $S$ is a $Z$-algebra) and use Azumaya's theorem to deduce that there exists a separable $Z$-algebra $T$ with $T\subseteq S$ such that $S=T+J(S)$. We may show that $J(T)=pT$ by observing that $pT$ is a nilpotent ideal, and that $T/pT$ is a separable (and hence semisimple) $Z/pZ$-algebra. Since $J(S)\subseteq J(R)$, we see that $R=S+J(R)=T+J(R)$ and so $S=T$ by the minimality of $S$. Since $J(T)=pT$ we see that~(iii) follows. \\

We require some extra information about the subring $S$ in the theorem above. First, a result due to Clark gives much more information about the structure of $S$. Recall that, for integers $k$ and $r$, the \emph{Galois ring} $\mathrm{GR}(p^k,r)$ may be defined by
\[
\mathrm{GR}(p^k,r)=\mathbb{Z}_{p^k}[x]/(f(x)),
\]
where $\mathbb{Z}_{p^k}$ denotes the integers modulo $p^k$, and where $f(x)\in \mathbb{Z}_{p^k}[x]$ is a monic polynomial of degree $r$ which is irreducible modulo $p$. Note that (see~\cite[Section~3]{Raghavendran} for example) the isomorphism class of the ring $\mathrm{GR}(p^k,r)$ is determined by $k$, $r$ and $p$. Note also that $\mathrm{GR}(p,r)\cong \mathbb{F}_{p^r}$.

\begin{Lem}
\label{S_structure}
Let $S$ be a finite ring of $p$-power order with an identity. Then $S$ is a direct sum of full matrix rings over Galois rings if and only if $J(S) = pS$.
\end{Lem}
\begin{proof}
See Clark~\cite[Lemma~3]{Clark}.
\end{proof}

\begin{Cor}
\label{S_count}
The number of isomorphism classes of rings $S$ that are the coefficient ring of some finite non-nilpotent ring of cardinality $p^n$ is at most $n\, 2^{3n}$. In particular, there are $p^{\Oh(n)}$ choices for the isomorphism class of $S$.
\end{Cor}
\begin{proof}
We have $|S|=p^s$, where $1\leq s\leq n$, and so there are at most $n$ choices for $s$. Suppose $s$ is fixed. Theorem~\ref{S_exist} and Lemma~\ref{S_structure} together imply that $S$ is a direct sum of full matrix rings over Galois rings:
\[
S\cong S_1\oplus S_2\oplus\cdots \oplus S_t,\text{ where }S_i=M_{m_i\times m_i}(\mathrm{GR}(p^{k_i},r_i)),
\]
for some positive integers $t$, $m_i$, $k_i$ and $r_i$. Clearly the isomorphism class of $S$ is determined by the choice of these integers. Now,
$|S_i|=p^{s_i}$ where $s_i=m_i^2k_ir_i$. Since $S$ is a direct sum of the rings $S_i$, we see that
\[
s=\sum_{i=1}^t s_i\geq \sum_{i=1}^t k_i.
\]
In particular, $(k_1,k_2,\ldots,k_t,s+1-\sum_{i=1}^tk_i)$ is an ordered (positive) integer partition of $s+1$. There are $2^{s}$ ordered integer partitions of $s+1$, and so there are at most $2^s$ choices for the integers $t$ and $k_1,k_2,\ldots ,k_t$. Similarly, there are at most $2^s$ choices for the integers $r_1,r_2,\ldots ,r_t$ and, since
\[
s\geq \sum_{i=1}^t m_i^2\geq \sum_{i=1}^t m_i,
\]
there are at most $2^s$ choices for the integers $m_1,m_2,\ldots ,m_t$. Hence the number of choices for the isomorphism class of $S$ is at most
\[
n(2^{s})^3\leq n\, 2^{3n},
\]
as required.
\end{proof}

\begin{Thm}
\label{module_count}
Let $S$ be a non-trivial direct sum of full matrix rings over Galois rings. Suppose that $|S|\leq p^n$. The number of isomorphism classes of (not necessarily unital) left $S$-modules of cardinality at most $p^n$ is at most $p^{2n^2+3n+1}=p^{\Oh(n^2)}$.
The same statement holds for right $S$-modules.
\end{Thm}
\begin{proof}
We prove the theorem for left-modules. The proof for right-modules is identical (or we can observe that a matrix ring $S$ is isomorphic to its opposite ring, so there is a canonical isomorphism between left and right $S$-modules).

We begin with considering the special case when $S=M_{m\times m}(\mathrm{GR}(p^{k},r))$ for positive integers $m$, $k$ and $r$.

We claim that $S$ is $2$-generated (as a ring) in this special case. To see this, first note that $\mathrm{GR}(p^{k},r)$ is $1$-generated: any element $\zeta\in \mathrm{GR}(p^{k},r)$ whose image in the natural map onto the finite field $\mathrm{GR}(p,r)$ is primitive (or, more generally, lies in no proper subfields) will generate $\mathrm{GR}(p^{k},r)$. We then note that, writing $E_{i,j}\in S$ for the matrix with $(i,j)$ entry $1$ and all other entries~$0$, the ring $S$ is generated by $g_1$ and $g_2$ where $g_1=\zeta E_{1,1}$ and where $g_2$ is the cyclic permutation matrix defined by
\[
g_2=E_{1,2}+E_{2,3}+\cdots+E_{m-2,m-1}+E_{m-1,m}+E_{m,1}.
\]
This establishes our claim.

A (not necessarily unital) left $S$-module $V$ of cardinality at most $p^n$ is determined by its structure as an abelian group, together with the pair of maps in $\mathrm{End}_\mathbb{Z}(V)$ induced by the left action of each of $g_1$ and $g_2$ on $V$. (Here, $\mathrm{End}_\mathbb{Z}(V)$ is the set of abelian group homomorphisms from $V$ to itself.) Now, as an abelian group we find that
\[
V\cong\mathbb{Z}_{p^{a_1}}\oplus \mathbb{Z}_{p^{a_2}}\oplus \cdots \oplus \mathbb{Z}_{p^{a_r}},
\]
for some positive integers $a_1,a_2,\ldots a_r$ where $\sum_{i=1}^ra_i\leq n$. In particular, the sequence $(a_1,a_2,\ldots,a_r,n+1-\sum_{i=1}^ra_i)$ is an ordered partition of $n+1$ into positive integers, and so there are at most $2^n$ choices for these integers. Since any element of $\mathrm{End}_\mathbb{Z}(V)$ is specified by the images of $r$ generators for $V$, we see that
\[
|\mathrm{End}_\mathbb{Z}(V)|\leq |V|^r\leq (p^n)^n=p^{n^2},
\]
so there are at most $p^{n^2}$ choices for the action of each of $g_1$ and $g_2$ on $V$. Hence the number of left $S$-modules is at most $p^{2n^2+n}$, and the theorem follows in this case.

We now consider the general case, so
\begin{equation}
\label{Ssum}
S=S_1\oplus S_2\oplus\cdots\oplus S_t\text{ where }S_i=M_{m_i\times m_i}(\mathrm{GR}(p^{k_i},r_i))
\end{equation}
for some positive integers $t$, $m_i$, $k_i$ and $r_i$. Note that, since $|S|\leq p^n$, we have~$t\leq n$. 

Let $e_i\in S$ be the identity matrix in $S_i$, so
\[
e_ie_j=\begin{cases} 0&\text{ when }i\not=j,\\
e_i&\text{ when }i=j.
\end{cases}
\]
The identity element $1$ of $S$ is $\sum_{i=1}^t e_i$. The subring $e_iSe_i$ is the $i$th ring in the sum~\eqref{Ssum}, and so is isomorphic to $S_i$.

Let $V$ be a left $S$-module of cardinality at most $p^n$. Setting $V_0$ to be the kernel of the map $v\mapsto 1v$ on $V$, there is an additive decomposition of $V$ of the form
\[
V=V_0\oplus 1V=V_0\oplus e_1V\oplus e_2V\oplus\cdots\oplus e_tV.
\]
For $i\in\{1,2,\ldots,t\}$, write $V_i=e_iV$. For $i\in\{0,1,\ldots,t\}$, define the non-negative integer $v_i$ by $|V_i|=p^{v_i}$. Since $\sum_{i=0}^t v_i\leq n$, and since $t\leq n$, we see that $\sum_{i=0}^t (v_i+1)\leq 2n+1$. Hence the sequence $(v_0+1,v_1+1,\ldots ,v_t+1,2n+2-\sum_{i=0}^t)$ is an ordered integer partition of $2n+2$, and so there are at most $2^{2n+1}$ possibilities for the integers $v_i$.

Since $S$ acts trivially on $V_0$, the isomorphism class of the module $V_0$ is entirely determined by its abelian group structure. Since $|V|=p^{v_0}$, the argument used in the special case above shows that there are at most $2^{v_0}$ possibilities for $V_0$ once $v_0$ is fixed.

Let $i\in\{1,2,\ldots,t\}$ be fixed. When $i\not=j$, the subring $e_jSe_j$ of $S$ acts trivially on $V_i$. So $V_i$ is entirely determined as a left $S$-module by its left $e_iSe_i$-module structure. Since $e_iSe_i\cong S_i$, the special case of the theorem we have already established shows that there are at most $p^{2v_i^2+v_i}$ possibilities for the isomorphism class of $V_i$. 

So the number of isomorphism classes of left $S$-modules $V$ is at most
\[
2^{2n+1}2^{v_0}\prod_{i=1}^t p^{2v_i^2+v_i}\leq p^{2n^2+3n+1},
\]
since $\sum_{i=0}^tv_i\leq n$, and since $2\leq p$. Hence the theorem follows.
\end{proof}

\begin{Thm} \label{extension}
Let $r$ and $n$ be integers such that $0 \leq r \leq n$. Let $J$ be a nilpotent ring of cardinality $p^r$. The number of isomorphism classes of rings $R$ of cardinality $p^n$ with $J(R)\cong J$ is at most $n\, p^{7n^2+9n+2}=p^{\Oh(n^2)}$.
\end{Thm}
\begin{proof}
The theorem is clearly true when $r=n$, since $R\cong J$ in this case. So we may assume that $r<n$.

Let $R$ be a ring of cardinality $R$ with $J(R)\cong J$. By replacing $R$ by a suitable isomorphic copy, we may assume that $J(R)=J$. We aim to show that the isomorphism class of the ring $R$ is determined by a certain sextuple of algebraic structures. Indeed, we aim to show that the isomorphism class of $R$ is determined by: the isomorphism class of a coefficient ring $S$; the isomorphism classes of left and right $S$-modules corresponding to left and right multiplication of $S$ on $J$; two abelian group isomorphisms that identify the underlying sets of these $S$-modules with $J$; an abelian group isomorphism that determines how $pS$ embeds in $J$. We will then proceed to count the number of possibilities for these structures.

Let $\mathcal{S}$ be a set of representatives for isomorphism classes of coefficient rings for rings of order $p^n$. For each ring $S\in\mathcal{S}$, choose a fixed set $C_S$ consisting of one representative from each coset of $pS$ in $S$.

Choose a coefficient ring $U$ for $R$, and let  $S\in\mathcal{S}$ be isomorphic to $U$. So there exists an isomorphism $\theta:S\rightarrow U$. 
By Theorem~\ref{S_exist}(iv), we know that $pU\subseteq J$. So the restriction of $\theta$ to $pS$ is an injective abelian group homomorphism $\psi:pS\rightarrow J$.

Theorem~\ref{S_exist}(i) and~(iv) implies that every element of $R$ may be uniquely written in the form $\theta(s)+x$ where $s\in C_S$ is one of the coset representatives for $pS$ in $S$ chosen above and where $x\in J$.

Let $\mathcal{V}_S$ and $\mathcal{V}'_S$ be sets of representatives for the isomorphism classes of, respectively, all left $S$-modules of cardinality $p^r$ and all right $S$-modules of cardinality $p^r$. Since $J$ is a left ideal in $R$, left multiplication by $U$ makes $J$ into a left $U$-module. Using the isomorphism $\theta$ between $S$ and $U$, we see that $J$ is a left $S$-module. Let $V\in \mathcal{V}_S$ be isomorphic to the left $S$-module that arises in this way, and let $\phi:V\rightarrow J$ be the induced module isomorphism (which, in particular, is an isomorphism of abelian groups). Similarly, right multiplication gives rise to a right $S$-module $V'\in\mathcal{V}'_S$ and an isomorphism $\phi':V'\rightarrow J$. For all $s\in S$, $v\in V$ and $v'\in V'$,
\[
sv=\theta(s) \phi(v)\text{ and }v's=\phi'(v')\theta(s).
\]

We have shown that each ring $R$ gives rise to (at least one) sextuple $(S,\psi,V,\phi,V',\phi')$ where $S\in \mathcal{S}$, $V\in \mathcal{V}_S$ and $V'\in\mathcal{V}'_S$, where $\psi:pS\rightarrow J$ is an injective group homomorphism, and where the maps $\phi:V\rightarrow J$ and $\phi':V'\rightarrow J$ are isomorphisms of abelian groups.

We may form a ring $T$ that is isomorphic to $R$ by taking all formal sums $s+x$ with $s\in C_S$ and $x\in J$ as our underlying set, and defining addition and multiplication as follows. Let $s_1,s_2\in C_S$. Now, $s_1+s_2=s+y$ for some $s\in C_S$ and $y\in pS$, and $s_1s_2=s'+y'$ for some $s'\in C_S$ and $y'\in pS$. We define the sum of elements $s_1+x_1$ and $s_2+x_2$ in $T$ to be
\[
s+\psi(y)+x_1+x_2,
\]
and the product of these elements to be
\[
s'+\psi(y')+\phi(s_1\,\phi^{-1}(x_2))+\phi'(\phi'^{-1}(x_1)\, s_2)+x_1x_2.
\]
We see that $T\cong R$, via the isomorphism that maps $s+x$ to $\theta(s)+x$ for all $s\in C_S$ and $x\in J$. Since $T$ is defined only using the sextuple $(S,\psi,V,\phi,V',\phi')$, this information is sufficient to determine the isomorphism class of $R$.

It remains to count the number of possibilities for $(S,\psi,V,\phi,V',\phi')$. There are at most $n2^{3n}$ possibilities for $S$, by Corollary~\ref{S_count}. Once $S$ is fixed, there are at most $p^{2n^2+3n+1}$ possibilities for each of the $S$-modules $V$ and $V'$, by Theorem~\ref{module_count}. The functions $\psi$, $\phi$ and $\phi'$ are all abelian group homomorphisms between $p$-primary groups of order at most $p^n$. Such groups are generated by at most $n$ elements, and a homomorphism is determined by the images of a generating set, so the number of choices for each of $\psi$, $\phi$ and $\phi'$ is at most $(p^n)^n=p^{n^2}$. Hence the number of possibilities for our sextuple is at most 
\[
n2^{3n}(p^{2n^2+3n+1})^2(p^{n^2})^3\leq n\, p^{7n^2+9n+2}=p^{\Oh(n^2)}.
\]
Each sextuple is associated with at most one isomorphism class of rings $R$, and every ring $R$ is associated with at least one sextuple. So the theorem follows.
\end{proof}

\begin{Thm}
\label{general bound}
The number of isomorphism classes of rings of cardinality $p^{n}$ is $p^{\delta}$ where $\delta = \frac{4}{27}n^{3} + \Oh(n^{5/2})$.
\end{Thm}
\begin{proof} 
Theorem 2.2 of \cite{K & P} shows the number of $\mathbb{F}_{p}$-algebras of cube zero and dimension $n$ is $p^{\frac{4}{27}n^{3} + \Oh(n^2)}$. This provides the lower bound we need. So to prove the theorem it is sufficient to show that there are at most $p^{\frac{4}{27}n^{3} + \Oh(n^{5/2})}$ rings of order $p^{n}$.

The Jacobson radical $J(R)$ of a ring $R$ of cardinality $p^n$ is nilpotent (as $R$ is artinian), and has cardinality $p^r$ for some integer $r$ such that $0\leq r\leq n$. By Theorem~\ref{main nilpotent}, there are at most $p^{\frac{4}{27}r^3+\Oh(r^{5/2})}$ nilpotent rings of cardinality~$p^r$, and so there are at most $p^{\frac{4}{27}r^3+\Oh(r^{5/2})}$ choices for the isomorphism class of $J(R)$. Once the isomorphism class of $J(R)$ is fixed, Theorem~\ref{extension} shows that there are at most $n\, p^{7n^2+9n+2}$ choices for the isomorphism class of $R$. Hence the number of rings of cardinality $p^n$ is at most
\[
\sum_{r=0}^{n} n\, p^{7n^2+9n+2} p^{\frac{4}{27}r^3+\Oh(r^{5/2})} = p^{\frac{4}{27}n^3+\Oh(n^{5/2})},
\]
as required.
\end{proof}

The high-level lesson we might take from the theorem above is that the structure of non-nilpotent rings is very restricted: the leading term of the enumeration function is provided by nilpotent rings. Indeed, it may well be the case (though we are far from having a proof) that the proportion of rings of order $p^n$ that are non-nilpotent tends to $0$ as $n\rightarrow\infty$. The following theorem provides a result in this direction.

\begin{Thm} \label{not very nonnilpotent bound}
Let $f_s(n)$ be the number of isomorphism classes of rings $R$ such that $|R|=p^n$ and $R/J(R)\geq p^s$. 
There exists a positive real number $\sigma$ such that, when we set $s=\sigma \sqrt{n}$,
\[
\lim_{n\rightarrow\infty} f_{s}(n)/\fr(n)=0.
\]
\end{Thm}
We note that Kruse and Price~\cite[Theorem~5.10]{K & P} provide a weaker version of this result, setting $s=\varepsilon n$ for some (arbitrary) positive real number $\varepsilon$. 
\begin{proof}
Theorem~\ref{main nilpotent} shows that the number of isomorphism classes of nilpotent rings of cardinality $p^r$ is at most $p^{\frac{4}{27}n^{3} + \kappa n^{5/2}}$ for some positive constant~$\kappa$.
Let $\sigma$ be a positive real number so that $\frac{4}{9}\sigma>\kappa$, and define $s=\sigma\sqrt{n}$. Following the approach of the proof of Theorem~\ref{general bound}, we see that
\[
f_s(n)\leq\sum_{r=0}^{n-\lceil s\rceil} n\, p^{7n^2+9n+2} p^{\frac{4}{27}r^3+\kappa r^{5/2}}=p^\alpha,
\]
where
\[
\alpha\leq \tfrac{4}{27}(n-s)^3+\kappa n^{5/2}+\Oh(n^2)\\
=\tfrac{4}{27}n^3+(\kappa-\tfrac{4}{9}\sigma) n^{5/2}+\Oh(n^2).
\]
Since $\fr(n)\geq p^{\frac{4}{27}n^3+\Oh(n^2)}$ by \cite[Theorem~2.2]{K & P}, we see that
\[
f_s(n)/\fr(n)\leq p^{(\kappa-\frac{4}{9}\sigma)n^{5/2}+\Oh(n^2)}\rightarrow 0
\]
as $n\rightarrow\infty$, since $\frac{4}{9}\sigma>\kappa$.
\end{proof}

We remark that, even though it is possible that most rings are nilpotent, the number of non-nilpotent rings is nevertheless large:

\begin{Thm} \label{nonnilpotent bound}
The number of isomorphism classes of non-nilpotent rings of cardinality $p^{n}$ is $p^{\frac{4}{27}n^{3} + \Oh(n^{5/2})}$.
\end{Thm}
\begin{proof}
The upper bound is provided by Theorem~\ref{general bound}. The lower bound may be proved by considering rings that are a direct sum of the form $\mathbb{F}_p\oplus N$, where $N$ is a nilpotent ring of order $p^{n-1}$. There are $p^{\frac{4}{27}n^3+\Oh(n^{5/2})}$ choices for the isomorphism class of $N$, by Theorem~\ref{main nilpotent}, and so the theorem follows.
\end{proof}

Finally we remark that, since the rings constructed for the lower bound all have an identity element, the following theorem holds:
\begin{Thm} \label{ring with one bound}
The number of isomorphism classes of rings with identity that have cardinality~$p^{n}$ is $p^{\frac{4}{27}n^{3} + \Oh(n^{5/2})}$.
\end{Thm}

\section{Finite commutative rings}
\label{sec:commutative}

This section contains proofs of our theorems on the enumeration of finite commutative rings. Some preliminary results (Theorem~\ref{thm:commutative_nilpotent_lower} and Lemma~\ref{lemma:commutative_Sims}) are easily proved by adapting the proofs of their more general non-commutative versions. For example, in a ring $R$ with $r$ generators $x_{1}, \ldots, x_{r}$, we usually need to consider $r^{2}$ possible products $x_{i}x_{j}$, whereas when $R$ is commutative, there are at most $\frac{1}{2}r(r + 1)$ such products that are distinct. However, in order to prove Theorem~\ref{thm:commutative_nilpotent}, a more thorough modification of the proof of Theorem~\ref{main nilpotent} is needed. 

We use the following theorem for our lower bound. 

\begin{Thm}
\label{thm:commutative_nilpotent_lower}
The number of (isomorphism classes of) commutative nilpotent $\mathbb{F}_{p}$-algebras with cube zero and dimension $n \geq 2$ is $p^{\alpha}$, where $\alpha = \frac{2}{27}n^{3} + O(n^{2})$. 
\end{Thm}
\begin{proof}
The lower bound is a consequence of Poonen~\cite[Lemma~9.1]{Poonen} and~\cite[Theorem~9.2]{Poonen}, so it suffices to establish a corresponding upper bound.

Let $f(n, r)$ be the number of (isomorphism classes of) commutative $\mathbb{F}_{p}$-algebras $A$ such that $\dim A = n, \dim (A/A^{2}) = r$ and $A^{3} = 0$. By following the proof of~\cite[Theorem~2.1]{K & P}, but using the free \emph{commutative} $\mathbb{F}_{p}$-algebra $F$ of cube zero and rank $r$, and observing that $F^2$ has dimension $\frac{1}{2}r(r+1)$, we may prove the following. First, if $\frac{1}{2}r(r + 1) < n - r$, then $f(n, r) = 0$. Secondly, if $\frac{1}{2}r(r + 1) \geq n - r$, then
\begin{displaymath}
\tfrac{1}{2}r(r + 1)(n - r) - (n - r)^{2} - r^{2} \leq \log_{p}f(n, r) \leq \tfrac{1}{2}r(r + 1)(n - r) - (n - r)^{2} + n - r. 
\end{displaymath}
The theorem now follows by observing that the value of $\frac{1}{2}r^2(n-r)$ takes its maximum value of $2n^3/27$ when $r=2n/3$, and following the proof of~\cite[Theorem~2.2]{K & P}.
\end{proof}

\begin{Lem}
\label{lemma:subalgebra_tower}
Let $R$ be a commutative nilpotent $\mathbb{F}_p$-algebra such that $r=\dim (R/R^2)$, $s$ is the Sims dimension of $R$, $t=\dim(R^2)$ and $R^3=0$. There exist $x_1,\ldots x_r\in R$ such that $x_1+R^2,x_2+R^2,\ldots,x_r+R^2$ form a basis for $R/R^2$, such that $\langle x_1,x_2,\ldots,x_s \rangle^2=R^2$ and such that for $1\leq i\leq s-1$
\[
x_ix_{i+1}\notin \langle x_1,x_2,\ldots,x_i\rangle^2.
\]
In particular,
\[
\langle x_1\rangle^2\subset \langle x_1,x_2\rangle^2\subset \cdots \subset \langle x_1,x_2,\ldots,x_s \rangle^2
\]
is a strictly increasing sequence of subalgebras.
\end{Lem}
\begin{proof}
By definition of the Sims dimension, there is a subalgebra $S$ of $R$ such that $\dim S/R^2=s$ and such that $S^2=R^2$. Moreover, we may assume that no proper subalgebra $T$ of $S$ has $T^2=R^2$. The existence of $x_1,x_2,\ldots,x_s\in R$ with the properties we want follows from~\cite[Proposition~10.1]{Poonen} (where, in the notation of Proposition~10.1, we take $V=S/R^2$, $W=R^2$, and multiplication as our symmetric bilinear map from $V\times V$ to $W$). The elements $x_1,x_2,\ldots,x_s$ are linearly independent modulo $R^2$, so there exist $x_{s+1},x_{s+2},\ldots,x_r\in R$ so that  $x_1+R^2,x_2+R^2,\ldots,x_r+R^2$ form a basis for $R/R^2$. So the lemma follows.
\end{proof}

Here is a commutative version of our Lemma~\ref{alpha}.

\begin{Lem}
\label{lemma:commutative_Sims}
Let $r, s, t$ be positive integers and let $\alpha(r, s, t)$ be a real number. Suppose that $p^{\alpha(r, s, t)}$ is the number of (isomorphism classes of) commutative $\mathbb{F}_{p}$-algebras, $R$, such that $r = \dim(R/R^{2})$, $s$ is the Sims dimension of $R$, $t = \dim(R^{2})$ and $R^{3} = 0$. Then 
\begin{align}
\alpha(r, s, t) & \leq  \tfrac{1}{2}r^{2}(t - s) + \mathnormal{O}\big((r + t)^{8/3}\big)  \label{old inequal} \\
 \textit{and} \hspace{5 mm} \alpha(r, s, t) & \leq  \tfrac{1}{2}r^{2}(t - s) + \tfrac{1}{2}rst + \mathnormal{O}\big((r + t)^{5/2}\big). \label{new inequal}
\end{align}
\end{Lem}
\begin{proof}
The inequality~\eqref{old inequal} is Poonen~\cite[Proposition~10.4]{Poonen}. The proof of~\eqref{new inequal} closely follows the proof of~\eqref{new_inequal} in Lemma~\ref{alpha} above. The sole change is to observe that we only require knowledge of the $\frac{1}{2}r(r+1)$ products $x_ix_j$ where $1\leq i\leq j\leq r$ since we are working in a commutative ring so, in the notation of the proof of Lemma~\ref{alpha},
\[
\alpha(r,s,t)\leq \binom{g}{2}d(t-d+1)+\tfrac{1}{2}r(r+1)d\leq \tfrac{1}{2}r^{2}(t - s) + \tfrac{1}{2}rst + \mathnormal{O}\big((r + t)^{5/2}\big).\qedhere
\]
\end{proof}

As in Section~\ref{Nilpotent case}, we choose a representative $R$ of each isomorphism class of commutative $\mathbb{F}_p$-algebras of cube zero described in Lemma~\ref{lemma:commutative_Sims}, and then we choose a \emph{standard basis} $x_1,x_2,\ldots,x_r,y_1,y_2,\ldots ,y_t$ of $R$. As before, the elements of the standard basis should have the properties that: the elements $x_1+R^2,x_2+R^2,\ldots,x_r+R^2$ form a basis for $R/R^2$; the elements $y_1,y_2,\ldots,y_t$ form a basis for $R^2$; each element $y_i$ may be written in the form $y_i=x_kx_\ell$ for some $k,\ell\in\{1,2,\ldots,s\}$. In addition, we require the following two properties. First, we require that
\[
\langle x_1\rangle^2\subset \langle x_1,x_2\rangle^2\subset \cdots \subset \langle x_1,x_2,\ldots,x_s \rangle^2
\]
is a strictly increasing sequence of subalgebras. Secondly, writing $q_i=\dim \langle x_1,x_2,\ldots,x_i \rangle^2$, we require that $y_1,y_2,\ldots ,y_{q_i}$ is a basis of $\langle x_1,x_2,\ldots,x_i \rangle^2$. A standard basis exists, by Lemma~\ref{lemma:subalgebra_tower}. Note that, since $q_1<q_2<\cdots<q_s=t$, we find that
\begin{equation}
\label{q}
q_i\leq t-s+i\text{ for }1\leq i\leq s.
\end{equation} 

For each element $y_j$, we choose a representation of the form $y_j=x_kx_\ell$ for some $k,\ell\in\{1,2,\ldots,s\}$, and call this the \emph{standard monomial representation} of $y_j$.
Clearly we may insist that $k\leq \ell$ here. Moreover, when $q_{i-1}<j\leq q_{i}$ we may (and we do) choose $\ell=i$.

\begin{Thm}
\label{thm:commutative_nilpotent}
The number of (isomorphism classes of) commutative nilpotent rings of cardinality $p^{n}$ is $p^{\alpha}$, where
$
\alpha = \frac{2}{27}n^{3} + \mathnormal{O}(n^{5/2}).
$
\end{Thm}
\begin{proof}
From Theorem~\ref{thm:commutative_nilpotent_lower}, the number of commutative $\mathbb{F}_{p}$-algebras with cube zero and order $p^{n}$ is $p^{\alpha'}$, where $\alpha' = 2n^{3}/27 + \mathnormal{O}(n^{2})$. Hence it is sufficient to prove that $p^{\alpha}$ is an upper bound for the number of commutative nilpotent rings of order $p^{n}$.

As in the proof of Theorem~\ref{main nilpotent}, we choose a concrete abelian group $G$ of order $p^n$ and rank $w$ that will be isomorphic to the additive group of $R$. We set $\overline{R}=R/pR$, so $\overline{R}$ is an $\mathbb{F}_p$-algebra of order $p^w$, and choose $r=\dim (\overline{R}/\overline{R}^2)$, $t=\dim (\overline{R}^2/\overline{R}^3)$, $u=\dim(\overline{R}^3)$ and the Sims dimension $s$ of $\overline{R}$ (which is equal to the Sims dimension of $\overline{R}/\overline{R}^3$). We choose $m$, the least integer such that $\overline{R}^m=0$, and we choose integers $u_h=\dim (\overline{R}^h/\overline{R}^{h+1})$.
%We define $d(i)=\dim(\overline{R}^i)$ for $i\geq 2$.
The argument in the proof of Theorem~\ref{main nilpotent} shows we have made $p^{O(n)}$ choices so far. We choose one of the $p^{\alpha(r,s,t)}$ isomorphism classes for the $\mathbb{F}_p$-algebra $\overline{R}/\overline{R}^3$.

For a nilpotent ring $R$ of cardinality $p^n$ that respects the choices we made above, we use a standard basis for $\overline{R}/\overline{R}^3$ to construct a basis
\[
\overline{x}_1,\overline{x}_2, \ldots,\overline{x}_r,\overline{e}_1,\overline{e}_2,\ldots,\overline{e}_{u+t} 
\]
for $\overline{R}$ where each element $\overline{e}_j$ is a monomial in $\overline{x}_1,\overline{x}_2,\ldots,\overline{x}_s$ just as in the proof of Theorem~\ref{main nilpotent}.  
In particular, $\overline{e}_{u-u_3+1},\overline{e}_{u-u_3+2},\ldots,\overline{e}_{u}$ form a basis for $\overline{R}^3$ modulo $\overline{R}^4$, and each element
$\overline{e}_{u-u_3+k}$ has the form $\overline{x}_{a}\overline{e}_{b}$ where $1\leq a\leq s$ and $u< b\leq u+t$. We claim that we may insist in addition that
$u< b\leq u+q_a$. To prove the claim, it suffices to show that the set
\[
S=\{\overline{x}_{a}\overline{e}_{b}:1\leq a\leq s\text{ and }u< b\leq u+q_a\}
\]
spans $\overline{R}^3$ modulo $\overline{R}^4$. The argument in the proof of Theorem~\ref{main nilpotent} shows that the larger set
\[
\hat{S}=\{\overline{x}_{a}\overline{e}_{b}:1\leq a\leq s\text{ and }u< b\leq u+t\} 
\]
spans $\overline{R}^3$ modulo $\overline{R}^4$. An element in $\hat{S}\setminus S$ may be written in the form $\overline{x}_{a}\overline{e}_{b}$ where $b>u+q_{a}$. We show that such an element is a sum of elements of $S$ modulo $\overline{R}^4$, which will establish our claim. Since $b>u+q_{a}$, we see that $\overline{e}_b=\overline{x}_{f}\overline{x}_{g}$ with $f\leq g$ and $g>a$. Using the fact that our ring is commutative,
\[
\overline{x}_{a}\overline{e}_{b}=\overline{x}_{a}\overline{x}_{f}\overline{x}_{g}=\overline{x}_{g}(\overline{x}_{f}\overline{x}_{a}).
\]
Since $a<g$ and $f\leq g$, we see that $\overline{x}_{f}\overline{x}_{a}+\overline{R}^3$ is contained in $\langle{x_1},x_2,\ldots,x_{g}\rangle^2+\overline{R}^3$. So,
working modulo $\overline{R}^3$, the product $\overline{x}_{f}\overline{x}_{a}$ is a sum of elements $\overline{e}_j$ with $u< j\leq u+q_{g}$. Hence, working modulo $\overline{R}^4$, the product $\overline{x}_{g}(\overline{x}_{f}\overline{x}_{a})$ is a sum of elements in $S$, and our claim follows.

We construct an additive generating set $x_1,x_2,\ldots,x_r,e_1,e_2,\ldots ,e_{u+t}$ for $R$ using the monomial representations of $\overline{e}_1,\overline{e}_2,\ldots ,\overline{e}_{u+t}$ just as before. We identify the elements of this generating set with elements of $G$: this requires making $p^{O(n^2)}$ choices.

The argument in Theorem~\ref{main nilpotent} shows that the isomorphism class of $R$ is determined once we know the products $x_ix_j$ with $1\leq i,j\leq r$ and the products $x_ie_j$ for $1\leq i\leq s$ and $1\leq j\leq u+t$ (as elements of $G$). Since $R$ is commutative, we only need to know the products $x_ix_j$ with $1\leq i\leq j\leq r$ and the products $x_ie_j$ for $1\leq i\leq s$ and $1\leq j\leq u+t$. Unfortunately, if we count the number of possibilities for these equations using the methods in Theorem~\ref{main nilpotent}, we do not get a sufficiently tight bound: we produce a bound of the form $p^{c n^3+O(n^{5/2})}$ with $c>2/27$. We improve this bound by showing that we do not need to know all of the products $x_ie_j$. Indeed we claim that multiplication in our ring is determined once we know the following products (as elements of $G$):

\begin{itemize}
\item[(i)] the products $x_ix_{j}$  for all $i, j$ such that $1 \leq i \leq j \leq r$,
\item[(ii)] the products $x_{i}e_{j}$ for all $i, j$ such that $1 \leq i \leq s$ and $u +1\leq j \leq u + q_i$.
\item[(iii)] the products $x_{i}e_{j}$ for  $1 \leq i \leq s$ and $1 \leq j \leq u$.
\end{itemize}

To prove this claim, we first observe, just as in Theorem~\ref{main nilpotent},  that these equations allows us to express every element $e_j$ as a monomial in $x_1,x_2,\ldots,x_s$. (Our more careful choice of $\overline{e}_{u-u_3+1},\overline{e}_{u-u_3+2},\ldots,\overline{e}_{u}$ is required at this point.) We show that the products $\overline{x}_{i}\overline{e}_{j}$ for $1\leq i\leq s$ and $1\leq j\leq u+t$ are determined by (i) to (iii) above (together with the additive structure $G$ of $R$). This suffices to prove our claim, using the argument in the proof of Theorem~\ref{main nilpotent}.   Indeed, writing $\mathrm{char}(R) = p^\kappa$, we will show that the products $p^ax_{i}e_{j}$ are determined for $0\leq a\leq \kappa$ by induction on $\kappa-a$ (the case $a=0$ giving us what we need). Clearly $p^ax_{i}e_{j}=0$ when $a=\kappa$ (and so $p^ax_{i}e_{j}$ is determined). Now suppose, as an inductive hypothesis, that $0<a\leq\kappa$ and the products $p^ax_{i}e_{j}$ are determined by~(i) to~(iii). We show that $p^{a-1}x_{i}e_{j}$ is determined. When $j\leq u+q_i$ the product is determined by the equations~(ii) and~(iii), so we may assume that $j>u+q_i$.  We may write $e_j=x_{a'}x_{b'}$ where $a'\leq b'$ and $b'>i$ (and this expression is determined by~(i)). Since $R$ is commutative, $p^{a-1}x_ie_j=p^{a-1}x_{b'}(x_ix_{a'})$. Since $a'\leq b'$ and $i<b'$ we find that $\overline{x}_i\overline{x}_{a'}\in\langle\overline{x}_1,\overline{x}_2,\ldots,\overline{x}_{b'}\rangle^2$, and so 
\[
x_ix_{a'}\in \langle e_1,e_2,\ldots,e_{u+q_{b'}},pe_{u+q_{b'}+1},\ldots,pe_{u+t}\rangle.
\]
Since multiplication is distributive and the product $x_ix_{a'}$ is determined by~(i), we see that $p^{a-1}x_ie_j$ is determined once we know $p^{a-1}x_{b'}x$ where
\[
x\in\{e_1,e_2,\ldots,e_{u+q_{b'}},pe_{u+q_{b'}+1},\ldots,pe_{u+t}\}.
\]
But the products
\[
p^{a-1}x_{b'}e_{1},p^{a-1}x_{b'}e_{2},\ldots,p^{a-1}x_{b'}e_{u+q_{b'}}
\]
are determined by the equations~(ii) and~(iii), and the products
\[
p^{a-1}(px_{b'}e_{u+q_{b'}+1}),p^{a-1}(px_{b'}e_{u+q_{b'}+2}),\ldots p^{a-1}(px_{b'}e_{u+t})
\]
are determined by our inductive hypothesis. Thus, since multiplication is distributive, $p^{a-1}\overline{x}_{i}\overline{e}_{j}$ is determined. So our result follows by induction. 

We have now shown that the isomorphism class of $R$ is fixed once we have chosen the products (i), (ii) and (iii) above. We now provide an upper bound for the number of these choices.

Each product $x_ix_j$ lies in $R^2$, and is already determined modulo $R^3+pR$ because we have chosen $\overline{x}_i$ as a standard basis. Since $|pR| = p^{n-w}$ and $|\overline{R}^3|=p^u$ we see that $|R^3+pR|=p^{n-w+u}=p^{n-r-t}$ and so there are at most $p^{n-r-t}$ choices for the product $x_ix_j$. There are $\frac{1}{2}r(r+1)$ products of the form~(i), and so the number of choices for for these products is at most $p$ to the power $\frac{1}{2}r(r+1)(n-r-t)$. Hence the number of choices for the products~(i) is at most $p$ to the power $\frac{1}{2}r^2(n-r-t)+O(n^2)$. 

For a fixed integer $i$ with $1\leq i\leq s$, there are $q_i$ choices of integers $j$ such that $u+1\leq j\leq u+q_i$. The product $x_ie_j$ lies in $R^3$, and $|R^3|\leq |R^3+pR|=p^{n-r-t}$. Therefore the number of choices for the equations~(ii) is at most $p$ to the power $\sum_{i=1}^sq_i(n-r-t)$. Using~\eqref{q}, we see that
\begin{align*}
\sum_{i=1}^sq_i(n-r-t)&\leq \sum_{i=1}^s(t-s+i)(n-r-t)\\
&=\left(s(t-s)+\tfrac{1}{2}s(s+1)\right)(n-r-t).
\end{align*}
So the number of choices for the products~(ii) is at most $p$ to the power $(st-\tfrac{1}{2}s^2)(n-r-t)+O(n^2)$.

A product $x_ie_j$ of the form~(iii) lies in $pR+\langle{e_1,e_2,\ldots,e_{j-1}}\rangle$, which is a subgroup of $G$ of order $p^{n-w+(j-1)}$. So the number of choices for the products~(iii) is at most $p$ to the power $s\sum_{j=1}^u(n-w+j-1)$. Since $w=r+t+u$, the number of choices for the products~(iii) is at most $p$ to the power $su(n-r-t)-\frac{1}{2}su^2+O(n^2)$.

To summarise, we have shown that the number of isomorphism classes of nilpotent commutative rings of order $p^n$ is at most $p^{\delta+O(n^2)}$, where $\delta$ is the maximum value of
\[
\alpha(r,s,t)+\tfrac{1}{2}r^2(n-r-t)+(st-\tfrac{1}{2}s^2)(n-r-t)+su(n-r-t)-\tfrac{1}{2}su^2
\]
over all non-negative integers $r,s,t,u$ such that $r+t+u\leq n$, $r\geq 1$ and $s\leq\min(r,t+1)$. Since $\delta$ is an increasing function of $u$, we may assume that $r+t+u=n$ in our maximisation.

We bound $\delta$ using standard techniques from calculus, just as in the proof of Theorem~\ref{main nilpotent}. Indeed, when $r\leq (3/5)n$ we may use the upper bound~\eqref{old inequal} for $\alpha(r,s,t)$ and show that $\delta\leq (9/125)n^3+O(n^{8/3})$. Since $9/125<2/27$, this implies that $\delta\leq (2/27)n^3+O(n^{5/2})$ when $r\leq (3/5)n$. When $r\geq (3/5)n$ we may use the upper bound~\eqref{new inequal} for $\alpha(r,s,t)$ to show that $\delta\leq (2/27)n^3+O(n^{5/2})$. So the theorem follows.
\end{proof}

Using Theorems~\ref{extension} and~\ref{thm:commutative_nilpotent} and mimicking the proof of Theorem~\ref{general bound} yields the following general result.

\begin{Thm}
\label{commutative main}
The number of isomorphism classes of commutative rings of cardinality $p^{n}$ is $p^{\delta}$, where $\delta = \frac{2}{27}n^{3} + O(n^{5/2})$.
\end{Thm}

Theorems~\ref{not very nonnilpotent bound}, \ref{nonnilpotent bound} and~\ref{ring with one bound} all have commutative analogues: 

\begin{Thm} \label{commutative not very nonnilpotent bound}
Let $g(n)$ be the number of isomorphism classes of commutative rings of order $p^n$.
Let $g_s(n)$ be the number of isomorphism classes of commutative rings $R$ such that $|R|=p^n$ and $R/J(R)\geq p^s$. 
There exists a positive real number $\sigma$ such that, when we set $s=\sigma \sqrt{n}$,
\[
\lim_{n\rightarrow\infty} g_{s}(n)/g(n)=0.
\]
\end{Thm}

\begin{Thm} \label{commutative nonnilpotent bound}
The number of isomorphism classes of commutative non-nilpotent rings of cardinality $p^{n}$ is $p^{\frac{2}{27}n^{3} + \Oh(n^{5/2})}$.
\end{Thm}

\begin{Thm} \label{commutative ring with one bound}
The number of isomorphism classes of commutative rings with identity that have cardinality~$p^{n}$ is $p^{\frac{2}{27}n^{3} + \Oh(n^{5/2})}$.
\end{Thm}

We omit the proofs of these theorems, since they are straightforward modifications of the proofs in Section~\ref{General case}. 

\appendix

\section*{Appendix: Some counterexamples}

In this appendix we provide counterexamples to Kruse and Price's proposed reduction to $\mathbb{F}_p$-algebras~\cite[Theorem~3.1]{K & P} , which also apply to the argument of Poonen~\cite[Lemma~11.1]{Poonen}. After explaining some details of the method used, Examples~1 to~4 below illustrate problems with the proof. The most serious problem is illustrated by Example~2 below: A process that is claimed to transform a ring $R$ of cardinality $p^n$ into an $n$-dimensional $\mathbb{F}_p$-algebra does not, in fact,  always produce an associative object. Since this process produces objects that are not $\mathbb{F}_p$-algebras, we do not see how to use it to reduce the enumeration of rings to the enumeration of $\mathbb{F}_p$-algebras. So we are forced to produce a proof that avoids this reduction.

\subsection*{The reduction to $\mathbb{F}_p$-algebras}

Theorem 3.1 of \cite{K & P} states that if the number of pairwise non-isomorphic $p$-algebras ($\mathbb{F}_p$-algebras in our terminology) of dimension $n$ is $p^{\alpha}$, then the number of pairwise non-isomorphic rings of order $p^{n}$ is less than $p^{\alpha + n^{2} + n}$. The proof begins as follows. Let $R$ be a ring of order $p^{n}$. Let $x_{1}, \ldots, x_{m}$ be a basis for the additive group of $R$, 
with char $x_{i} = p^{k_{i}}$ for $1 \leq i \leq m$, so that $\sum_{i=1}^{m}k_{i} = n$. Next, for $1 \leq i \leq m$ and $0 \leq j < k_{i}$, define $y_{ij} = p^{j}x_{i}$. (Here, as in \cite{K & P book} and \cite{Poonen}, minor misprints in \cite{K & P} have been corrected.) Rename the $y_{ij}$ in any convenient (predetermined) ordering as $z_{1}, \ldots, z_{n}$. Then there are integers $\phi_{ijk}$ for $1 \leq i, j, k \leq n$, such that $0 \leq \phi_{ijk} < p$ and
\begin{displaymath}
z_{i}z_{j} = \sum_{k=1}^{n}\phi_{ijk}z_{k}.
\end{displaymath}
Moreover, the ring $R$ is determined up to isomorphism by the structure constants $\phi_{ijk}$. \par Now define an $\mathbb{F}_p$-algebra $A$ with a basis $e_{1}, \ldots, e_{n}$ by setting
\begin{displaymath}
e_{i}e_{j} = \sum_{k=1}^{n}\phi_{ijk}e_{k}.
\end{displaymath}
Kruse and Price assert that ``By construction the multiplicative semigroups of $A$ and $R$ are isomorphic, so associativity of $A$ is equivalent to associativity of $R$''. This is not true, as Examples~1 and~2 below demonstrate. For completeness, we give (Example~3) a ring $R$ with two different bases that give rise to non-isomorphic algebras $A$, and (Example~4) two non-isomorphic rings that give rise to the isomorphic algebras when appropriate bases are chosen. These examples answer natural questions about this process of obtaining an algebra from a ring. Examples~3 and~4 are also counterexamples to statements in the proofs of~\cite[Theorem~3.1]{K & P} and~\cite[Lemma~11.1]{Poonen} (that are not ultimately problematic). \\

\subsection*{Four examples}

\noindent
\textbf{Example 1}. (To show that the multiplicative semigroups of $R$ and $A$ need not be isomorphic.) Let  $R = \mathbb{Z}/8\mathbb{Z}$, the ring of integers modulo 8, then the group of units of $R$ is the direct sum of two cyclic groups, each of order 2. On the other hand, it is easy to verify that $A \cong F_{2}[X]/(X^{3})$, whose group of units is cyclic of order 4, generated by $(1 + X) + (X^{3})$. \\

\noindent
\textbf{Example 2}. (To show that $A$ need not be associative.) Writing $\mathbb{Z}_9$ for $\mathbb{Z}/9\mathbb{Z}$, let $R$ be the ring $\mathbb{Z}_9[X]/I$ where $I=(X^2-3)$. The additive group of $R$ is the sum of two cyclic groups, each of order $9$. A natural basis of $R$ is $1+I,X+I$, but we instead choose the basis $x_1,x_2$, where $x_1=5+X+I$ and $x_2=2+2X+I$. 
\par Now take $z_{1} = x_{1}, z_{2} = x_{2}, z_{3} = px_{1} = 3x_{1}$ and $z_{4} = px_{2} = 3x_{2}$. Then
\begin{align*}
z_{1}^{2} & =  1+X+I =  2z_{2} + z_{4}, \\
z_{1}z_{2} = z_2z_1 &=7+3X +I =  z_{1} + z_{2},\\
z_{2}^{2} & =  7+8X +I= 2z_{1} + z_{4}, \text{ and}\\
z_{2}z_{4}  = z_{4}z_2  &= 3+6X+I= 2z_{3}.
\end{align*}
So our $\mathbb{F}_3$-algebra $A$ has a basis $e_1,e_2,e_3,e_4$ such that
\begin{align*}
e_{1}^{2} &  =2e_{2} + e_{4}, \\
e_{1}e_{2} = e_2e_1 &=  e_{1} + e_{2},\\
e_{2}^{2} & = 2e_{1} + e_{4}, \text{ and}\\
e_{2}e_{4}  = e_{4}e_2  &= 2e_{3}.
\end{align*}
We calculate
\[
(e_{1}e_1) e_{2} = (2e_{2} + e_{4})e_{2} = 2e_{2}^{2} + e_{4}e_{2} = 2(2e_{1} + e_{4}) + 2e_{3}
=  e_{1} + 2e_{3} + 2e_{4},
\]
and also
\[
e_{1}(e_{1}e_{2})= e_{1}(e_{1} + e_{2}) = e_{1}^{2} + e_{1}e_{2} = (2e_{2} + e_{4}) + (e_{1} + e_{2})
= e_{1} + e_{4}.
\]
Hence $(e_{1}e_1) e_{2}\neq e_{1}(e_{1}e_{2})$, and so $A$ is not associative. \\

\noindent
\textbf{Example 3}. (To show that choosing different bases can produce non-isomorphic algebras $A$.) If we choose the natural basis $1+I,X+I$ for the ring $R$ in the example above, it is not hard to show that $A \cong \mathbb{F}_{3}[X]/(X^{4})$. This is very different from the non-associative algebra constructed in our previous example. \\

\noindent
\textbf{Example 4}. (To show that two non-isomorphic rings $R$ and $R_1$ can give rise to the same algebra $A$.) Define $R$ as in Examples~2 and~3. Define $R_1=\mathbb{Z}_9[X]/(X^2-6)$. It is not hard to check that the basis $1+(X^2-6),X+(X^2-6)$ of $R_1$ gives rise to the $\mathbb{F}_p$-algebra $A\cong \mathbb{F}_{3}[X]/(X^{4})$ just as in Example~3. But the rings $R_1$ and $R$ are not isomorphic: writing $e_1$ and $e$ for their respective identities, a short calculation shows that $R$ contains no element $y$ such that $y^{2} = 6e$, but $R_1$ contains the element $y=X+(X^2-6)$ with the property that $y^2=6e_1$.

\paragraph{Acknowledgement} We would like to thank the referee for their careful reading of the paper, and their insightful and helpful comments, which significantly improved our exposition.

\end{document}